\documentclass[12pt, letterpaper]{article}
\usepackage[utf8]{inputenc}
\usepackage[margin=2.54cm]{geometry}
\usepackage{amsmath, amsthm, amssymb, bbm}
\usepackage{hyperref}
\usepackage{enumitem}
\usepackage{graphicx}
\usepackage{standalone}
\usepackage{mathrsfs}
\usepackage{mathtools}
\usepackage{verbatim}
\usepackage{filecontents}

\def\bx{{\bf x}}
\newcommand{\brac}[1]{\left(#1\right)}

\newcommand{\set}[1]{\left\{#1\right\}}
\newtheorem{theorem}{Theorem}
\newtheorem{conjecture}[theorem]{Conjecture}
\newtheorem*{theorem*}{Theorem}
\newtheorem{lemma}[theorem]{Lemma}

    \newtheoremstyle{TheoremNum}
{\topsep}{\topsep}
{\itshape}
{}
{\bfseries}
{.}
{ }
{\thmname{#1}\thmnote{ \bfseries #3}}
\theoremstyle{TheoremNum}

\newcommand{\mb}{\mathbb}
\newcommand{\mc}{\mathcal}

\newcommand{\1}{\mathbbm{1}}
\author{Tolson Bell\thanks{thbell@cmu.edu. Research supported by NSF Graduate Research Fellowship grant DGE 2140739.}~ and Alan Frieze\thanks{frieze@cmu.edu. Research supported in part by NSF grant DMS1952285.}\\Carnegie Mellon University}
\date{August 27, 2023}
\title{\vspace{-1cm}Giant Rainbow Trees in Sparse Random Graphs}

\begin{document}
\maketitle
\begin{abstract}
For any small constant $\epsilon>0$, the Erd\H os-R\'enyi random graph $G(n,\frac{1+\epsilon}{n})$ with high probability has a unique largest component which contains $(1\pm O(\epsilon))2\epsilon n$ vertices. Let $G_c(n,p)$ be obtained by assigning each edge in $G(n,p)$ a color in $[c]$ independently and uniformly. Cooley, Do, Erde, and Missethan proved that for any fixed $\alpha>0$, $G_{\alpha n}(n,\frac{1+\epsilon}{n})$ with high probability contains a rainbow tree (a tree that does not repeat colors) which covers $(1\pm O(\epsilon))\frac{\alpha}{\alpha+1}\epsilon n$ vertices, and conjectured that there is one which covers $(1\pm O(\epsilon))2\epsilon n$. In this paper, we achieve the correct leading constant and prove their conjecture correct up to a logarithmic factor in the error term, as we show that with high probability $G_{\alpha n}(n,\frac{1+\epsilon}{n})$ contains a rainbow tree which covers $(1\pm O(\epsilon\log(1/\epsilon)))2\epsilon n$ vertices.
\end{abstract}
\section{Introduction}
\subsection{The Giant Component}
In their seminal 1960 paper, Erd\H os and R\'enyi showed that for constant $\epsilon>0$, a random graph with more than $\frac{1+\epsilon}{n}\binom n2$ edges contains a unique component of size $\Theta(n)$, which they called the giant component. In particular, we know the asymptotic size of this component: 
\begin{theorem}[\cite{ErdosRenyi,bollobasrandomgraphs}]\label{giant}
Let $\epsilon>0$ constant, $m = (1+\epsilon)n/2$. Then with high probability, $G(n,m)$ consists of a unique giant component, with $(1-\frac{\mu}{1+\epsilon}\pm o(1))n$ vertices, where $\mu$ is the solution in $(0,1)$ of the equation $\mu e^{-\mu}=(1+\epsilon)e^{-(1+\epsilon)}$.
\end{theorem}
In this theorem and throughout the paper, \textit{with high probability}, or \textit{w.h.p.}, refers to a probability that goes to 1 as $n$ goes to infinity with $\epsilon>0$ fixed. Similarly, $o(1)$ refers to a probability that goes to 0 as $n\rightarrow\infty$. On the other hand, even though $\epsilon$ is fixed with $n$, we care about ``sufficiently small'' $\epsilon$, so we will still use notation like $O(\epsilon^2)$ to mean a term that is $O(1)$ in $n$ but is $O(\epsilon^2)$ as $\epsilon\rightarrow 0$.

Using Taylor series, we see that $\mu$ in Theorem \ref{giant} equals $1-\epsilon+O(\epsilon^2)$, and thus the giant component has size $(1-\frac{\mu}{1+\epsilon}\pm o(1))n=(1\pm o(1)\pm O(\epsilon))2\epsilon n$. For our purposes, we will be working with $G(n,p)$ not $G(n,m)$, but all of our results transfer between the two models, as standard Chernoff bounds give that the number of edges in $G(n,\frac{1+\epsilon}{n})$ is w.h.p.~$\frac{(1+\epsilon)n}{2}\pm O((n\log n)^{1/2})=(1\pm o(1))\frac{(1+\epsilon)n}{2}$.
\subsection{Rainbow Coloring}
Now, we consider the colored random graph, $G_c(n,p)$, where every edge that is chosen in $G(n,p)$ is then uniformly and independently given a color from $[c]$. A set of edges in $G_c(n,p)$ is considered \textit{rainbow} if each edge is present but no two edges in the set have the same color. Cooley, Do, Erde, and Missethan conjectured the following:
\begin{conjecture}[\cite{CDEM}, Conjecture 4.1]\label{conj}
Let $\alpha>0,c=\alpha n,\epsilon>0$ sufficiently small. Let $p=\frac{1+\epsilon}{n}$. Then w.h.p.~the largest rainbow tree in $G_c(n,p)$ has order $(1\pm O(\epsilon))2\epsilon n$.
\end{conjecture}
$(1\pm O(\epsilon))2\epsilon n$ is clearly an upper bound on the size of the largest rainbow tree, as the size of the largest tree in $G(n,p)$ is also $(1\pm O(\epsilon))2\epsilon n$. So the task is to show that there w.h.p.~exists a rainbow tree of this size. Cooley, Do, Erde, and Missethan also proved a weaker version of their conjecture:
\begin{theorem}[\cite{CDEM}, Theorem 1.4]
Let $\alpha>0,c=\alpha n,\epsilon>0$ sufficiently small. Let $p=\frac{1+\epsilon}{n}$. Then w.h.p.~the largest rainbow tree in $G_c(n,p)$ has order $(1\pm O(\epsilon))\frac{\alpha}{\alpha+1}\epsilon n$.
\end{theorem}
The result of this paper is that we improve the leading constant from $\frac{\alpha}{\alpha+1}$ to its optimal value of 2, resolving their conjecture up to a logarithmic factor in the lower-order term:
\begin{theorem}\label{main}
Let $\alpha>0,c=\alpha n,\epsilon>0$ sufficiently small. Let $p=\frac{1+\epsilon}{n}$. Then w.h.p.~the largest rainbow tree in $G_c(n,p)$ has order $(1\pm O(\epsilon\log(1/\epsilon)))2\epsilon n$.
\end{theorem}
\subsection{Relation to Previous Literature}

A major goal in $G_c(n,p)$ has been to find rainbow thresholds for given structures, that is, to prove that for certain $p$ and $c$, there w.h.p.~is (or is not) a rainbow copy of that structure. Previous results in this model have shown optimal or near-optimal rainbow thresholds for the existence of certain trees, which can also be thought of as showing the size of the rainbow connected component. Frieze and Mckay showed that $G_c(n,p)$ has a rainbow spanning tree w.h.p.~whenever the uncolored graph is connected and at least $n-1$ colors have appeared \cite{FriezeMckay}. Aigner-Horev, Hefetz, and Lahiri showed that for any given tree of size $(1-\epsilon)n$ for constant $\epsilon>0$, it suffices to take $p=\omega(1/n)$ and $c=n$ for there to w.h.p.~exist a rainbow copy of that tree \cite{bigrainbowtrees}. Cooley, Do, Erde, and Missethan showed that for any fixed $\epsilon>0$, there exists a $d(\epsilon)>1$ such that $G_n(n,\frac dn)$ will w.h.p. have a tree, and in fact a cycle, of size $(1-\epsilon)n$ \cite{CDEM}. Our paper's result can be thought of as showing the threshold for a rainbow tree of size $\epsilon n$ for small fixed $\epsilon>0$. For a tree of size $o(n)$ but $\omega(n^{2/3})$, the optimal result was also shown by Cooley, Do, Erde, and Missethan \cite{CDEM}.

In general, a breakthrough result on uncolored thresholds by Frankston, Kahn, Narayanan, and Park \cite{FKNP} was extended by Bell, Frieze, and Marbach to show that for any structure with $k$ edges, $p$ can always be taken to be within a $\log(k)$ factor of the optimal value as long as $c>\gamma k$ for some $\gamma>1$ \cite{BFM}. This could be applied to trees of size $\epsilon n$, but the logarithmic factor is too large.

Our proof relies heavily on previous results and distributional lemmas regarding the structure of the uncolored giant component, which we will discuss in Section 2. We will prove Theorem \ref{main} in Section 3.
\section{Structure of the Uncolored Giant Component}
\subsection{The Core and Mantle}
Throughout the paper, we will define $a\lesssim b$ to mean $a\le b(1\pm o(1)\pm O(\epsilon))$, that is, $\le$ up to factors allowed by Conjecture \ref{conj}. We will use $\gtrsim$ and $\approx$ similarly. For instance, we can say the giant component has size $\approx 2\epsilon n$.

After its size, the next thing about the giant component we will need to know is the size of its core. The \textit{core}, or \textit{2-core}, of a graph $G$ is the maximal subgraph of $G$ where all vertices have degree at least 2. Note that the core of a connected component is always itself connected.
\begin{lemma}[\cite{AKSlongpath};\cite{geomdepth}, Theorem 3;\cite{giantcore}, Theorem 1(ii);\cite{BollobasBook,AlanBook}]\label{2core}
Let $\epsilon>0$ constant and let $C$ be the core of the giant component in $G(n,\frac{1+\epsilon}{2})$. With high probability,\[|V(C)|\approx|E(C)|\approx 2\epsilon^2n.\]
\end{lemma}
In other words, only an $\epsilon$ fraction of the giant component lies in the core. The vertices and edges that are in the giant component but not in its core are called the \textit{mantle} vertices and edges. Note that the vertices of the giant component along with only the mantle edges form a forest where each component contains exactly one vertex in the core. 

A more specific description of the giant component was given in 2015 by Ding, Lubetzky, and Peres \cite{DING2014155}. They gave a model that the core of $G(n,\frac{1+\epsilon}{n})$ is \textit{contiguous} to, which means that any graph property is true w.h.p.~in their model is true w.h.p.~in the core of $G(n,\frac{1+\epsilon}{n})$. Recall that $\mu$ is the solution in $(0,1)$ to $\mu e^{-\mu}=(1+\epsilon)e^{-(1+\epsilon)}$. Their model is the following:
\begin{itemize}
    \item First, let $\Lambda\sim\mc{N}(1+\epsilon-\mu,\frac 1n)$. For $1\le u\le n$, let $D_u\sim$ Poisson($\Lambda$) iid, conditioned on $\sum D_u\1_{D_u\ge 3}$ being even. Let $N_k$ be the number of $u$ such that $D_u=k$, and let $N=\sum_{k\ge 3}N_k$. Select a random multigraph $\mc{K}$ on $N$ vertices uniformly among all multigraphs with $N_k$ vertices of degree $k$ for all $k\ge 3$.
    \item Then, replace each edge of $\mc{K}$ with paths of length iid $Geom(1-\mu)$ to form $\mc{K}'$.
    \item Finally, independently for each vertex $v$ in $\mc{K}'$, add on a Poisson($\mu$)-Galton-Watson tree rooted at $v$.
\end{itemize}
For our proof, we do not need any details of the first two steps (beyond that they form a core $C$ with $|V(C)|\approx|E(C)|\approx 2\epsilon^2n$). We heavily rely on more specific details of the final step, and for the rest of this paper will prove everything w.h.p.~under their model.
\subsection{Galton-Watson Trees}

The \textit{Poisson($\mu$)-Galton-Watson tree}, or \textit{$\mu$-PGW tree}, is a random tree that is created as follows:
\begin{itemize}
    \item Start with a root vertex, which we will call level 0. Set $j=0$.
    \item While level $j$ is not empty:
    \begin{itemize}
        \item Independently for each vertex $v$ in level $j$, create Poisson($\mu$) children of $v$, which will be on level $j+1$.
        \item Increment $j$.
    \end{itemize}
\end{itemize}Because $\mu<1$, the resulting tree is finite with probability 1. This model was studied by Bienaym\'e in 1845 and Galton and Watson in 1875 \cite{Bienayme,GaltonWatson}. In 1942, Borel proved that the number of vertices (including the root) in a $\mu$-PGW tree $T$ follows the following distribution: for any $k\in\mb{Z}_{>0}$,
\[\mb{P}(|V(T)|=k)=\frac{e^{-\mu k}(\mu k)^{k-1}}{k!},\] which is called the Borel($\mu$) distribution \cite{Borel,DING2014155}.

For any edge $e$ in a $\mu$-PGW tree, let $desc(e)$ be the number of vertices in the subtree rooted at the child vertex of $e$. Note that each vertex created then becomes the root of a new $\mu$-PGW tree, so $desc(e)\sim$ Borel($\mu$) for every $e$ that is created. Now that we know that the Borel distribution appears in the giant component, we will bound its tail probabilities.
\begin{lemma}\label{sqrttail}
Let $X\sim $Borel($\mu$). For any $j\in\mb{Z}_{>0}$, \[\mb{P}(X>j)<\frac{1}{j^{1/2}\mu}.\]
\end{lemma}
\begin{proof}
\begin{align*}
\mb{P}(X>j)&=\sum_{k=j+1}^\infty\frac{e^{-\mu k}(\mu k)^{k-1}}{k!}<\sum_{k=j+1}^\infty\frac{e^{-\mu k}(\mu k)^{k-1}}{2k^{1/2}(k/e)^k}=\sum_{k=j+1}^\infty\left(e^{1-\mu}\mu\right)^k(2\mu)^{-1}k^{-1.5}\\&<(2\mu)^{-1}\sum_{k=j+1}^\infty k^{-1.5}<(2\mu)^{-1}\int_{k=j}^\infty k^{-1.5}=\frac{1}{j^{1/2}\mu}
\end{align*}
\end{proof}
\subsection{The Mantle of the Giant Component}
Let $M$ be the set of mantle edges. For any $e\in M$, note that $desc(e)$ (with respect to the $\mu$-PGW tree containing $e$ rooted in the core) equals the number of vertices in the mantle that would be disconnected from the core if $e$ were removed. From the last section we know that for a fixed $e\in M$, $desc(e)$ is distributed as Borel($\mu$), where $\mu=1-\epsilon\pm O(\epsilon^2)$. 
\begin{lemma}\label{concentrated}
Let $1\le j\le O(1)$ with respect to $n$ (but $j$ can depend on $\epsilon$) and let $D_j=|\{e\in M:desc(e)>j\}|$. Then, 
\[
\mb{P}\brac{D_j>\frac{3\epsilon n}{j^{1/2}}}<n^{-2}.
\]
\end{lemma}
\begin{proof}
When creating the mantle under the Ding--Lubetzky--Peres model, for each edge $e$ created, we know $\mb{P}(desc(e)>j)=\mb{P}(Borel(\mu)>j)\le 1.1/j^{1/2}$ for sufficiently small $\epsilon$ by Lemma \ref{sqrttail}. We also know w.h.p.~that $|M|\lesssim 2\epsilon n$, so $\mb{E}D_j\lesssim 2.2\epsilon n/j^{1/2}$. To show concentration of $D_j$ around its expectation, we use the following theorem from Warnke \cite{War}:
\begin{theorem}\label{warnke}
Let $X=(X_1,X_2,\ldots,X_N)$ be a family of independent random variables with $X_k$ taking values in a set $\Lambda_k$. Let $\Omega=\prod_{k\in[N]}\Lambda_k$ and suppose that $\Gamma\subseteq \Omega$ and suppose that $f:\Omega\to{\bf R}$ are given. Suppose also that whenever $\bx,\bx'\in \Omega$ differ only in the $k$th coordinate 
\[
|f(\bx)-f(\bx')|\leq \begin{cases}c_k&if\ \bx\in\Gamma.\\d_k&otherwise.\end{cases}
\]
If $W=f(X)$, then for all reals $\gamma_k>0$,
\[
\mb{P}(W\geq \mb{E}(W)+t)\leq \exp\set{-\frac{t^2}{2\sum_{k\in[N]}((c_k+\gamma_k(d_k-c_k)^2))}}+\mb{P}(X\notin \Gamma)\sum_{k\in [N]}\gamma_k^{-1}.
\]
\end{theorem}
We apply this theorem with $W=D_j$ and $X_k$ equal to the set of neighbors of vertex $k$ in $[k-1]$. We take $\Gamma$ to be the occurence of (i) the maximum degree in $G_{n,p}$ is at most $n^\eta$ for some small $\eta$ and (ii) every tree in $G_{n,p}$  of size at most $n^{1-\eta}$ that contains at most one vertex (the root) that is adjacent to vertices outside the tree, has size at most $n^\eta$. Under these circumstances ${\bf Pr}(X\notin \Gamma)\leq e^{-n^\eta}$ and $c_k\leq n^{2\eta}$ and $d_k\leq n$. (This is proved by first moment calculations. In particular, property (ii) follows from the gap in component sizes, see Lemma 2.14 of \cite{AlanBook}.) We can then take $t=n^{1/2+2\eta},\gamma_k=n^{-3}$ in the theorem to show that $W\leq 3\epsilon/j^{1/2}$ with the required probability.

\end{proof}
\section{Proof of Theorem \ref{main}}
First, create $G=G(n,p)$, and let $E$ be the set of edges in the giant component. Choose an ordering $e_1,\dotsc,e_{|E|}$ that satisfies the following:
\begin{itemize}
    \item All core edges come before all mantle edges
    \item The mantle edges are ordered in descending order by their number of descendents, that is, $desc(e_i)\ge desc(e_{i+1})$ for any $e_i,e_{i+1}\in M$.
\end{itemize}
Now, go in this order, and when you get to $e_i$, color it independently uniformly at random from $[\alpha n]$ (or technically, $[\lceil\alpha n\rceil]$, but we will omit ceilings for convenience). If the color of $e_i$ has already been used on some $e_j$ for $j<i$, delete $e_i$.

This process results in a rainbow colored graph that is a subgraph of the giant component. Our claim is that this new graph has a component of size $(1-O(\epsilon\log(1/\epsilon))2\epsilon n$, which would clearly prove the theorem, as any rainbow component has a rainbow spanning tree. In other words, we are allowed to delete at most $O(\epsilon^2\log(1/\epsilon)n)$ vertices from the giant component.

\begin{lemma}
After we process the first $5\epsilon^2n$ edges ($e_1,\dotsc,e_{5\epsilon^2n}$), w.h.p. the new graph $G'$ has a component of size $\approx2\epsilon n$.
\end{lemma}
\begin{proof}
When each edge $e_i$ gets colored, there are less than $i$ already used colors, out of $\alpha n$ possible colors. Thus, for each edge, the probability that it gets deleted is less than $i/(\alpha n)$, so for every $i\le 5\epsilon^2n$, $e_i$ gets deleted with probability less than $5\epsilon^2/\alpha$. Thus, in this process, every edge in $G$ is kept with probability at least $1-\frac{5\epsilon^2}{\alpha}$, independently of what has happened to the previous edges. This gives a direct coupling from $G'$ to $G(n,p')$ where \[p'=p\left(1-\frac{5\epsilon^2}{\alpha}\right)=\frac{1+(1-O(\epsilon))\epsilon}{n}\] for fixed $\alpha$ and sufficiently small $\epsilon$. This coupling is such that every edge that appears in $G(n,p')$ also appears in $G'$. Then we apply Theorem \ref{giant} to say that $G(n,p')$ has a giant component of size $\approx 2\epsilon n$.
\end{proof}
\begin{lemma}\label{edgei}
With high probability, in the original graph $G$, we have that $e_i$ is in the mantle and $desc(e_i)\le 36(\epsilon n/i)^2$ for every $5\epsilon^2n\le i\le |E|$.
\end{lemma}
\begin{proof}
Lemma \ref{2core} tells us that w.h.p.~for sufficiently small $\epsilon$, there are less than $5\epsilon^2n/2$ edges in the core and less than $5\epsilon n/2$ edges total,  which we can assume these throughout the proof. This implies that $e_i$ is in the mantle for every $5\epsilon^2n/2\le i\le |E|<5\epsilon n/2$.

Fix $5\epsilon^2n\le i\le |E|$. Noting that $5\epsilon^2n\le i$ implies that $36(\epsilon n/i)^2$ is $O(1)$ with respect to $n$, Lemma \ref{concentrated} tells us that with probability at least $1-n^{-2}$,
\[|\{e\in M:desc(e)>36(\epsilon n/i)^2\}|<\frac{3\epsilon n}{6(\epsilon n/i)}<\frac{i}{2}\] 
Thus with probability at least $1-n^{-2}$, there are less than $5\epsilon^2n/2+i/2\le i$ total edges that are in the core or have more than $36(\epsilon n/i)^2$ descendants, and thus $desc(e_i)\le36(\epsilon n/i)^2$.

Taking a union bound over all $5\epsilon^2n\le i\le |E|$ gives us that with probability at least $1-(5\epsilon n/2)(n^{-2})$, or w.h.p., we have that $desc(e_i)\le36(\epsilon n/i)^2$ for every $5\epsilon^2n\le i\le |E|$.
\end{proof}
Applying Lemma \ref{edgei} with $i=5\epsilon^2n$ tells us that in the new graph $G'$, as the core of $G'$ is a subset of the core of $G$, any uncolored edge is either in the mantle or has been disconnected from the core. If an uncolored edge is still in the mantle, it has the same number of descendants in $G'$ as it did in $G$, as all of its descendant edges are also uncolored (as they have fewer descendants). So for uncolored edges, we will still use $desc(e_i)$ to equal what it was in $G$, and then deleting any future $e_i$ disconnects at most $desc(e_i)$ vertices from the rainbow giant (specifically, $desc(e_i)$ if $e_i$ was previously in the mantle and 0 if it has already been disconnected). These properties remain true as the process continues.

Now, for every $5\epsilon^2n<i\le |E|$, let $X_i$ equal $desc(i)$ if $e_i$ is deleted and $X_i=0$ if $e_i$ is kept. Then note that the total number of vertices disconnected from the core in our entire process is at most $O(\epsilon^2n)+\sum_{i=5\epsilon^2n+1}^{|E|}X_i$.
\begin{lemma}\label{ExpDisconnect}
For every $5\epsilon^2n<i\le |E|$, $\mb{E}(X_i|X_1,\dotsc,X_{i-1})<36\epsilon^2n/(i\alpha)$ for every possible sequence $X_1,\dotsc,X_{i-1}$.
\end{lemma}
\begin{proof}
Lemma \ref{edgei} gives that for every $5\epsilon^2n\le i\le |E|$, $desc(e_i)\le 36(\epsilon n/i)^2$. Then no matter the values of $X_1,\dotsc,X_{i-1}$, the probability that $e_i$ gets a banned color is less than $\frac{i}{\alpha n}$. Thus 
\[
\mb{E}(X_i|X_1,\dotsc,X_{i-1})<\frac{i}{\alpha n}\left(36\left(\frac{\epsilon n}{i}\right)^2\right)=\frac{36\epsilon^2n}{i\alpha}.
\]
\end{proof}
The proof of Theorem \ref{main} follows from the Chernoff-Hoeffding bound. Due to Lemma \ref{ExpDisconnect}, we can treat each $X_i$ as an independent random variables with mean $36\epsilon^2n/(i\alpha)$ and maximum value $36(\epsilon n/i)^2\le 2\epsilon^{-2}$. If $S=\sum_{i=5\epsilon^2n}^{|E|}X_i$ then, where $H_k=\sum_{j=1}^k1/j$,
\[
\mb{E}(S)=\sum_{i=5\epsilon^2n}^{|E|}\frac{36\epsilon^2n}{i\alpha}=\frac{36\epsilon^2n}{\alpha}(H_{|E|}-H_{5\epsilon^2n})<\frac{36\epsilon^2n}{\alpha}\left(\log\left(\frac{e^{1/2}|E|}{5\epsilon^2n}\right)\right)<\frac{36\epsilon^2\log(1/\epsilon)n}{\alpha}
\]
Then 
\[
\mb{P}\left(S-\mb{E}S>\frac{\epsilon^2\log(1/\epsilon)n}{\alpha}\right)<\exp\left(\frac{-2(\epsilon^2\ln(1/\epsilon)n/\alpha)^2}{(3\epsilon n)(2\epsilon^{-2})^2}\right)=\exp\left(\frac{-\epsilon^5\ln(1/\epsilon)n}{6\alpha^2}\right)=o(1)
\]
So w.h.p.~$S\le37\epsilon^2\log(1/\epsilon)n/\alpha$. Thus, the coloring ends with a rainbow giant component of size at least  $(1-37\epsilon\log(1/\epsilon)/\alpha-O(\epsilon))2\epsilon n$, completing the proof.
\subsection{Lower Bound on Vertices Lost}
The conjecture of Cooley, Do, Erde, and Missethan allows the rainbow tree to have $O(\epsilon^2n)$ fewer vertices than the uncolored component, while ours loses $O(\epsilon^2\log(1/\epsilon)n)$. One may wonder if there is a rainbow tree losing even fewer vertices, say $O(\epsilon^3n)$ or $o(n)$; however, the conjectured allowance is the best possible. We know from the Borel distribution and something analogous to Lemma \ref{concentrated} that there are w.h.p.~$\approx 2\epsilon n/e$ leaf vertices in the giant component. If we randomly order these $\approx 2\epsilon n/e$ leaf edges and color them in order with $\alpha n$ colors, each edge has probability $\approx\frac{\epsilon}{e\alpha}$ of being given a repeated color. So we will w.h.p.~need to delete $\gtrsim2\epsilon^2/e^2\alpha=\Theta(\epsilon^2n)$ leaves when going from the uncolored to rainbow giant component.
\bibliographystyle{plain}
\bibliography{GiantRainbow}
\end{document}